\documentclass{amsart}
\usepackage{amssymb}
\usepackage{graphics}
\usepackage{latexsym}
\usepackage{amsmath}
\usepackage{amssymb}
\usepackage{amscd}
\usepackage{multicol}
\usepackage[cmtip,matrix,arrow]{xy}

\newtheorem{thm}{Theorem}%[section]
\newtheorem{prop}[thm]{Proposition}

\theoremstyle{definition}

\theoremstyle{remark}

\newcommand{\RR}{\mathbb R}

\newcommand{\rk}{\mathrm{rank}\, }

\newcommand{\mfk}{\mathfrak{k}}

\newcommand{\mfh}{\mathfrak{h}}

\newcommand{\mfu}{\mathfrak{u}}

\begin{document}
\address[O.~Goertsches]{Fachbereich Mathematik und Informatik der Philipps-Universit\"at 
\newline Marburg\\
Germany}
\address[A.-L.~Mare]{Department of Mathematics and Statistics\\ University of Regina\\ Canada}

\email[]{goertsch@mathematik.uni-marburg.de}

\email[]{mareal@math.uregina.ca}

\title[Topological cohomogeneity-one actions]{Equivariant cohomology of  cohomogeneity-one actions:
the topological case}

\author{Oliver Goertsches}
\author{Augustin-Liviu Mare}

\begin{abstract}
We show that for any cohomogeneity-one continuous action of a compact connected Lie group $G$ on a closed topological  manifold the equivariant cohomology equipped with its canonical 
$H^*(BG)$-module structure is  Cohen-Macaulay.  The proof relies on the structure theorem for these actions recently obtained by  Galaz-Garc\'{\i}a and Zarei.  We generalize in this way our previous result
concerning smooth actions.  
\end{abstract}  
\maketitle
%\tableofcontents

Let $M$ be a closed topological manifold, $G$ a compact connected Lie group, and 
$G\times M \to M$ a continuous action of cohomogeneity one, i.e., such that the orbit space
$M/G$ is one-dimensional. A complete description of such group actions has been obtained
only recently by Galaz-Garc\'{\i}a and Zarei in \cite{GZ}. As a consequence, 
they were able to identify within the setup above the actions which do not fit into the smooth category. We recall that smooth cohomogeneity-one actions  had been previously dealt with by  Mostert in \cite{Mo}.
    
Our main goal here is to prove the following result. (The coefficient ring for cohomology is  always  $\RR$.)

\begin{thm}\label{main}
Let $M$ and $G$ be  as above, such that the action has cohomogeneity equal to 1.
Then the equivariant cohomology group $H^*_G(M)$ equipped with its canonical structure of a module over $H^*(BG)$ is Cohen-Macaulay. 
\end{thm}

We  recall that $H^*(BG)$ is a polynomial ring in several variables. 
Furthermore, a non-zero module over any such ring is Cohen-Macaulay if it is finitely
generated and its depth is equal 
to its Krull dimension (see, e.g., \cite[Sect.~1.5]{Bruns-Herzog}).

In the special case when the $G$-action is smooth, the theorem above has been proved in \cite{Go-Ma}. 
We will try to keep the present paper self-contained; however, for the sake of brevity, sometimes we will prefer to simply invoke results already proved in the aforementioned work. 
%We start with a technical lemma.

%\begin{lem} \label{lemasubg} If $G$ is a compact, possibly non-connected Lie group and $G'\subset G$ a normal subgroup, then  $H^*(BG) \simeq H^*(B(G/G')) \otimes H^*(BG')$ by a canonical  isomorphism of vector spaces.  \end{lem}

%\begin{proof}  Write $BG =(EG \times E(G/G'))/G$ where the group action involved here is diagonal, $G$ acting on $E(G/G')$ via the canonical projection $G\to G/G'$. The projection onto the second factor gives rise to a fibration  $BG' \to BG \to B(G/G').$ Since all three components of the fibration have vanishing odd-dimensional cohomology, by an elementary spectral sequence argument one deduces the required isomorphism. \end{proof}

%Thanks to the previous lemma,  we do not lose any generality in regards to our main result if we restrict ourselves to effective actions.

The structure theorem of \cite{GZ, Mo} says as follows. 
Let $M$ and $G$ as indicated above. Then the orbit space $M/G$ is homeomorphic to either
the circle $S^1$  or  the interval $[0,1]$. 
In the first case there is just one orbit type, say $(H)$, and $M$ is a $G/H$-bundle over $S^1$.
If $M/G=[0,1]$   
 then $M$ is determined  by a {\it group diagram}.
This  is of the form $(G, H, K^-, K^+)$ where $H,K^-,K^+$ are
closed subgroups of $G$ such that $H\subset K^{\pm}$ and  
each of $K^+/H$ and $K^-/H$ is homeomorphic to a sphere or to the Poincar\'e homology sphere
${\bf P}^3$.   The cohomogeneity one manifold encoded by this quadruple is $$ M = G \times_{K^-}C(K^-/H) \cup_{G/H}
G \times_{K^+} C(K^+/H),$$ where $C(K^\pm/H) =(K^\pm/H \times [0,1])/(K^\pm/H \times \{1\} =*)
$ are the cones over $K^\pm/H$. The $G$-action has three isotropy types, which are represented by
$H$ (the principal one), $K^-$, and $K^+$, respectively. 
% The projection $G \times_{K^\pm}C(K^\pm/H)\to G/K^{\pm}$ is a bundle whose fiber is the cone $C(K^{\pm}/H)$. Thus  $G \times_{K^\pm}(K^\pm/H\times (0,1])$ is open in $G \times_{K^\pm}C(K^\pm/H)$. The two subspaces $G \times_{K^\pm}C(K^\pm/H)$ of $M$ have in common $G \times_{K^\pm}(K^\pm/H\times \{0\})=G/H$. 

The following result will be needed later; it is similar in spirit to \cite[Proposition 3.1]{Go-Ma}. 

\begin{prop}\label{prorank} Let $K$ be a compact Lie group, possibly non-connected, which acts transitively on  the Poincar\'e homology sphere ${\bf P}^3$ and let $H\subset K$ be an isotropy subgroup.
Then $\rk H =\rk K -1$ and the canonical homomorphism 
$H^*(BK)\to H^*(BH)$ is surjective.
\end{prop}

\begin{proof} 
If the $K$-action is effective, then, by \cite{Br}, $K$ is a rank 1 Lie group and $H$ is a finite group, hence the ranks satisfy the required equation
 (note that in \cite{Br} it is assumed that $K$ is connected: however, if this not the case, observe that the action of the identity    component of $K$ on ${\bf P}^3$ is transitive as well). Hence both claims in the proposition hold true in this particular situation. 
 
Let us now assume that the kernel of the action, call it $K'$, is a non-trivial (normal) subgroup of  $K$.
 The desired conclusion about the ranks follows by writing  \newline $K/H = (K/K')/(H/K')$ and noticing that $\rk K = \rk (K/K') + \rk K'$ and $\rk H = \rk (H/K') + \rk K'$,
see, e.g., \cite[Ch.~9, Prop.~2 (c)]{Bou}. 
It  remains to justify the surjectivity claim. The method we use is inspired by the proof of Prop.~3.1, Case 2, in \cite{Go-Ma}. 
 The two cohomology rings involved in the statement can be identified with the
invariant rings $S(\mfk^*)^K$ and $S(\mfh^*)^H$, where $\mfk$ and $\mfh$ are the Lie algebras of $H$ and $K$ respectively,
and $S(\mfk^*)$, $S(\mfh^*)$ the symmetric algebras of their duals, see \cite[p.~311]{Mi-St}.
 We need to show that the map $S(\mfk^*)^K\to S(\mfh^*)^H$ induced by restriction from $\mfk$ to $\mfh$ is surjective.
  Start by noticing that $\mfh$ is also the Lie algebra of $K'$, since $H/K'$ is a finite group.  
  Thus $\mfh$ is an ideal in $\mfk$. Relative to an ${\rm Ad}K$-invariant inner product we consider
  the orthogonal decomposition $\mfk =\mfh \oplus \mfu$ which expresses $\mfk$ as  a direct sum of Lie algebras
  (the complement $\mfu$ being isomorphic to ${\mathfrak s}{\mathfrak u}(2)$).      
  Let $f:\mfh \to {\mathbb R}$ be an arbitrary $H$-invariant polynomial function. Consider $g: \mfk \to {\mathbb R}$,
  $g(X+Y):=f(X)$, for all $X\in \mfh$ and $Y\in \mfu$. To show that $g$ is $K$-invariant, we first
  point out that $K$ is generated as a group by $H$ and the identity component $K_0$ of $K$ (the exact homotopy sequence
  of the fibration $H \to K \to {\bf P}^3$ shows that any connected component of $K$ contains at least one component of $H$).
  Both $H$ and $K_0$ leave $\mfh$ invariant and consequently preserve the splitting $\mfk=\mfh \oplus \mfu$.
  Once we show that $f$ is $K_0$-invariant the proof will be finished. But this 
  follows from the fact that $[\mfu, \mfh]=0$. 
\end{proof}

We are now in a position to prove the main result.

\noindent {\it Proof of Theorem \ref{main}.} 
If $M/G=S^1$ then, as already mentioned, all isotropies are conjugate to each other. The $G$-action on $M$ is Cohen-Macaulay by \cite[Corollary 4.3]{Go-Ro}, see also \cite[Corollary 1.2]{Fr}. 
 The same argument works also in the case when $M/G = [0,1]$ and $\rk H = \rk K^-=\rk K^+$.
Let us now assume that $M/G=[0,1]$ and $\rk H\le \rk K^+\le \rk K^-$ but  
$\rk H < \rk K^-$.   
Denote $b:=\rk K^-$. By \cite[Proposition 3.1]{Go-Ma} and Proposition \ref{prorank} above, $\rk H = b-1$.   
Consider the open $G$-invariant covering $M = U \cup V$ with
$U := G \times_{K^-}C(K^-/H) \cup_{G/H}
G \times_{K^+} (K^+/H\times [0, \epsilon))$ and 
$V :=G \times_{K^-}(K^-/H\times [0, \epsilon)) \cup_{G/H}
G \times_{K^+} C(K^+/H)$ 
where $\epsilon>0$ is small. Note that $U$ and $V$ contain $G/K^-$ and $G/K^+$ respectively
as $G$-equivariant deformation retracts. Thus the corresponding Mayer-Vietoris sequence is
$$\cdots \longrightarrow H^*_G(M) \longrightarrow H^*_G(G/K^-) \oplus H^*_G(G/K^+)
\longrightarrow H^*_G(G/H) \longrightarrow \cdots .$$
By \cite[Proposition 3.1]{Go-Ma} and Proposition \ref{prorank}, the last map in the sequence above 
is surjective, hence the sequence splits:
\begin{equation}\label{seq}0 \longrightarrow H^*_G(M) \longrightarrow H^*(BK^-) \oplus H^*(BK^+)
\longrightarrow H^*(BH) \longrightarrow 0 .\end{equation} 
If both $K^-/H$ and $K^+/H$ are spheres, the proof in \cite[Sect.~4]{Go-Ma} goes through without any modification. Assume now that at least one of $K^-/H$ and $K^+/H$ is homeomorphic to
${\bf P}^3$.
Recall that $b-1\le \rk K^+\le b$.

\noindent{\it Case 1:} $\rk K^+=b$. We use the same argument as in \cite[Sect.~4, Case 1]{Go-Ma}. 
 
\noindent{\it Case 2:} $\rk K^+=b-1$. By Proposition \ref{prorank}, $K^+/H$ is is a sphere. Thus this time we  simply need to notice that everything in \cite[Sect.~4, Case 2]{Go-Ma} remains valid in this new setup.
\hfill $\square$
 
 As a byproduct we have proved that if $M/G$ is a closed interval and  \newline  
 $\rk H <{\rm max} \{\rk K^-,\rk K^+\}$ then 
$H^*_G(M)$ is isomorphic to the kernel of 
$$H^*(BK^-) \oplus H^*(BK^+) \to H^*(BH), (f,g)\mapsto \pi_1(f) -\pi_2(g),$$
where $\pi_1$ and $\pi_2$ are the obvious maps. 

In general, the Krull dimension of $H^*_G(M)$ over the ring $H^*(BG)$
is equal to the maximal rank of the isotropy subgroups. The action is equivariantly formal, i.e., the  module mentioned above is free, if and only if there is at least one isotropy subgroup of the same rank as $G$. To illustrate this principle, recall that the main difference between 
topological and smooth cohomogeneity-one actions is that only in the former case one can have 
$K^-/H={\bf P}^3$ or $K^{+}/H={\bf P}^3$; now, such an action can be equivariantly formal only if
$M$ is even dimensional (otherwise, by \cite[Sect.~5.2]{Go-Ma}, all isotropies would have maximal rank).    
A concrete situation is analyzed below.

 Consider the (non-smoothable) cohomogeneity-one action described in \cite[Example 2.3]{GZ}.
 The group diagram is  $(S^3\times {\rm SO}(n+1), I^*\times {\rm SO}(n),I^*\times {\rm  SO}(n+1),$ \newline $ S^3\times {\rm SO}(n))$,
 where   $I^*$  is the binary icosahedral group; the manifold acted on is ${\bf P}^3 * S^n$.
 Identify $H^*(BS^3)=\RR[u]$, where $\deg u =4$. 
 The equivariant cohomology is  the kernel of  the map 
 $H^*(B{\rm SO}(n+1)) \oplus \RR[u] \otimes H^*(B{\rm SO}(n)) \to H^*(B{\rm SO}(n)),$ $(f,g) \mapsto \pi_1(f)-\pi_2(g),$
 as a module over $\RR[u] \otimes H^*(B{\rm SO}(n+1))$. If $n$ is odd then there exists 
  $f\in \ker \pi_1$, $f\neq 0$; but then $f.(0,u)=0$,  hence the module is not free, in conformity with the previous discussion.  
  If $n$ is even, then  $\pi_1$ is injective and the equivariant cohomology is isomorphic to 
  $\pi_1(H^*(B{\rm SO}(n+1))) + u\RR[u] \otimes  H^*(B{\rm SO}(n))$, which is a subring of $\RR[u] \otimes H^*(B{\rm SO}(n))$. 
  This  module over $\RR[u] \otimes H^*(B{\rm SO}(n+1))$ is now free: a basis consists of $1$ and $ue$,
  where $e\in H^*(B{\rm SO}(n))$ is such that $\{1, e\}$ is a basis of $H^*(B{\rm SO}(n))$ over $H^*(B{\rm SO}(n+1))$.

 \noindent {\bf Acknowledgement.} We wish to thank the referee for suggesting an improvement.

\end{document}